\documentclass[reqno]{amsart}
\usepackage{amssymb}
\usepackage{hyperref}
\usepackage{booktabs}
 \usepackage{multirow}

\newtheorem{theorem}{Theorem}
\newtheorem{proposition}[theorem]{Proposition}

\theoremstyle{remark}

\newtheorem*{note}{Note}

\numberwithin{equation}{section}

\begin{document}

\title[Quadrature from Bernstein-Szeg\"o polynomials]
{Quadrature rules from finite orthogonality relations for Bernstein-Szeg\"o polynomials}

\author{J.F.  van Diejen}

\address{
Instituto de Matem\'atica y F\'{\i}sica, Universidad de Talca,
Casilla 747, Talca, Chile}

\email{diejen@inst-mat.utalca.cl}

\author{E. Emsiz}

\address{
Facultad de Matem\'aticas, Pontificia Universidad Cat\'olica de Chile,
Casilla 306, Correo 22, Santiago, Chile}
\email{eemsiz@mat.uc.cl}

\subjclass[2010]{Primary: 65D32 ;  Secondary 33C47, 33D45, 47B36}
\keywords{quadrature rules, Bernstein-Szeg\"o polynomials,
orthogonality relations, Jacobi matrices}

\thanks{This work was supported in part by the {\em Fondo Nacional de Desarrollo
Cient\'{\i}fico y Tecnol\'ogico (FONDECYT)} Grants  \# 1170179 and  \# 1181046.}

\date{April 2018}

\begin{abstract}
We glue two families of Bernstein-Szeg\"o polynomials to
construct the eigenbasis of an associated finite-dimensional Jacobi matrix. This gives rise to finite orthogonality relations for this composite eigenbasis of Bernstein-Szeg\"o polynomials. As an application, a number of Gauss-like quadrature rules are derived for the exact integration of rational functions with prescribed poles against the Chebyshev weight functions.
\end{abstract}

\maketitle



\section{Introduction}\label{sec1}
From the three-term recurrence relation for the Chebyshev polynomials
of the second kind \cite[Chapter 18]{olv-loz-boi-cla:nist}, it is immediate that these diagonalize a semi-infinite Jacobi matrix with zeros on the diagonal and units on the sub-- and superdiagonals.  Upon modifying the orthogonality measure via division by a positive polynomial of degree $d$, the Chebyshev basis  passes over into a basis of orthonormal
Bernstein-Szeg\"o polynomials \cite[Section 2.6]{sze:orthogonal}.
These polynomials turn out to diagonalize a ($d$-parameter family)  of semi-infinite Jacobi matrices that are perturbed at the top left block
\cite{sze:orthogonal,dam-sim:jost,ger-ili:bernstein-szego}:
\begin{equation*}
\begin{bmatrix}
b_0 & a_1 &   &   &  &       \\
a_1 &\ddots& \ddots&   &   &    \\
     &  \ddots  &  \ddots &  a_{l} &    &    \\
         && a_l &  b_l&  1&     \\
      & & &1&0&\ddots   \\
     &   & & &\ddots &\ddots   \\
\end{bmatrix} ,
\end{equation*}
where $l=\lfloor \frac{d}{2}\rfloor$ (and $b_l\equiv 0$ when $d$ is even).
In this note we glue two such families of Bernstein-Szeg\"o polynomials depending on $d$ and $\tilde{d}$ parameters, respectively, so as to diagonalize 
a corresponding $(d+\tilde{d})$-parameter family of $(m+1)$-dimensional Jacobi matrices of the form
\begin{equation*}
\begin{bmatrix}
b_0 & a_1 &   &   &  &  &   &   &  &  &      \\
a_1 & \ddots & \ddots &   &   & &   &   & &  &       \\

     &  \ddots  &  \ddots &  a_{l} &    &  & &  &   &   &   \\
         && a_l &  b_l&  1&    &  &  &   &   &  \\
      & & &1&0&\ddots &  &  & & &  \\
     &   & & &\ddots&\ddots &\ddots &  &  & &  \\
  &   &   & & &\ddots &0 &1 &  &  &   \\
  &   &   && & &1&b_{m-\tilde{l}}&a_{m+1-\tilde{l}} &  &    \\
   &   &   & & & & &a_{m+1-\tilde{l}} &\ddots & \ddots &    \\
    &   && & & &  &   & \ddots &\ddots & a_m  \\
     &   && & & & & & &a_m &b_m \\
\end{bmatrix} ,
\end{equation*}
where  $l=\lfloor \frac{d}{2}\rfloor$,  $\tilde{l}=\lfloor \frac{\tilde{d}}{2}\rfloor$, and with $m$ positive such that $m+1\geq\lceil  \frac{d}{2}\rceil+\lceil\frac{\tilde{d}}{2} \rceil $.
The symmetry of the the Jacobi matrix gives rise to a finite-dimensional system of discrete orthogonality relations for the pertinent composite eigenbasis built of Bernstein-Szeg\"o polynomials. By standard arguments (cf. e.g. \cite{sze:orthogonal,gau:survey,dav-rab:methods}),
these orthogonality relations imply in turn Gauss-like quadrature rules for the exact integration of rational functions with prescribed poles against the Chebyshev weight function.

When $\tilde{d}=d=2$, the present construction recovers orthogonality relations that lie at the basis of a four-parameter family of discrete Fourier transforms unifying all sixteen standard types of discrete (co)sine transforms DST-$k$ and DCT-$k$ ($k=1,\ldots ,8$) \cite{die-ems:discrete}, whereas for  $\tilde{d}=0$ with $d$ arbitrary one recuperates
discrete orthogonality relations for the Bernstein-Szeg\"o polynomials stemming from the Gauss quadrature rule
\cite{dar-gon-jim:quadrature,bul-cru-dec-gon:rational,die-ems:exact}.

The presentation splits up in four parts. First our main fundamental result is stated in Section \ref{sec2}. Specifically---after recalling some classical facts concerning the Bernstein-Szeg\"o polynomials---we introduce a finite grid consisting of nodes $\xi^{(m)}_0,\ldots ,\xi_m^{(m)}$ on the interval $[0,\pi]$ whose positions are governed by an elementary transcendental equation (cf. Eqs. \eqref{bethe:eq}, \eqref{ua}).
To each node $\xi^{(m)}_{\hat{l}}$ ($0\leq\hat{l}\leq m$), we can associate an $(m+1)$-dimensional vector
$\psi^{(m)}\bigl(  \xi^{(m)}_{\hat{l}} \bigl)= \bigl[ \psi_l^{(m)}\bigl(  \xi^{(m)}_{\hat{l}} \bigl) \bigr]_{0\leq l\leq m} $
with components $\psi_l^{(m)}\bigl(  \xi^{(m)}_{\hat{l}} \bigl)$ built of two families of Bernstein-Szeg\"o polynomials evaluated at the node (cf. Eq. \eqref{bs-basis}). Our main theorem states that the vectors 
$\psi^{(m)}\bigl(  \xi^{(m)}_{0} \bigl),\ldots ,\psi^{(m)}\bigl(  \xi^{(m)}_{m} \bigr)$
thus constructed satisfy an explicit system of orthogonality relations (cf. Theorem \ref{orthogonality:thm}). 

Section \ref{sec3} is devoted to the proof of these orthogonality relations, which proceeds in two steps. First, the orthogonality is verified by inferring that
(after performing a diagonal gauge transformation) the vectors in question provide an eigenbasis of an $(m+1)$-dimensional Jacobi matrix of the form displayed above, with entries built from the recurrence coefficients of the two underlying families of Bernstein-Szeg\"o polynomials. As a by-product, this reveals that our nodes parametrize the eigenvalues of the Jacobi matrix in question.
Next, the quadratic norms of the vectors are calculated by evaluating the diagonal of the corresponding Christoffel-Darboux kernel at the nodes.

Section \ref{sec4} collects some further miscellaneous results of interest. 
Specifically,  we compute the characteristic polynomial of our Jacobi matrix,  provide estimates for the locations of the nodes $\xi^{(m)}_0,\ldots ,\xi_m^{(m)}$ (and thus for the eigenvalues of the Jacobi matrix),  and exhibit---as a special example of our construction---finite orthogonality relations for the Askey-Wilson polynomials at $q=0$.

As a principal application, we wrap up in Section \ref{sec5}
with a precise description of the quadrature rules stemming from our orthogonality relations (cf. Theorem \ref{quadrature:thm}). Indeed, by `column-row duality' the orthogonality established in this note implies discrete orthogonality relations  for a finite system of Bernstein-Szeg\"o polynomials supported on the nodes $\xi^{(m)}_0,\ldots ,\xi_m^{(m)}$.
By comparing these discrete orthogonality relations with the conventional continuous orthogonality relations for the Bernstein-Szeg\"o polynomials,
the pertinent quadrature formulas are read-off immediately. A key feature of the quadrature rules under consideration is that these allow for the exact integration
of rational functions with prescribed poles (outside the integration interval) originating from the denominator of the orthogonality measure for the Bernstein-Szeg\"o polynomials.

\begin{note}
In principle the Bernstein-Szeg\"o  polynomials associated with the Chebyshev polynomials of the first--, the third-- and the fourth kind
may be viewed as parameter degenerations of the ones above stemming from the Chebyshev polynomials of the second kind \cite{sze:orthogonal,gri:oscillatory}.
However, to avoid limit transitions we have adopted the standard practice of
formulating the statements below uniformly for all four kinds.
\end{note}
 
\section{Statement of the main result}\label{sec2}

\subsection{Preliminaries on Bernstein-Szeg\"o polynomials}
The Bernstein-Szeg\"o polynomials  constitute an orthogonal basis for the Hilbert space $L^2\bigl( (0,\pi), w(\xi) \text{d}\xi\bigr)$, where the weight function is of the form
\begin{subequations}
\begin{equation}\label{bs-wf}
w (\xi) := \frac{1}{2\pi |{c}(\xi)|^{2}}  \qquad (0<\xi <\pi),
\end{equation}
with
\begin{equation}\label{c-f}
{c}(\xi)   := {(1+\epsilon_+e^{-i\xi})^{-1} (1-\epsilon_-e^{-i\xi})^{-1}}\prod_{1\leq r \leq d} (1+\alpha_r e^{-i\xi}) 
\end{equation}
\end{subequations}
(and $i:=\sqrt{-1}$). Here (and below) it is always assumed that
\begin{equation*}
\epsilon_\pm\in\{0,1\}\quad \text{and}  \quad   0< |\alpha_r| <1\quad (r=1,\ldots ,d) ,
\end{equation*}
with possible complex parameters $\alpha_r$ occurring in complex conjugate pairs. 
More specifically, the Bernstein-Szeg\"o basis $p_l(\xi)$, $l=0,1,2,\ldots$
enjoys a unitriangular expansion on the Fourier-cosine monomials
$1,(e^{i\xi}+e^{-i\xi}), (e^{2i\xi}+e^{-2i\xi}),\ldots $ and is obtained from them via Gram-Schmidt orthogonalization with respect to the weight function $w(\xi)$. Rather than to work with a monic-- or an  orthonormal basis,  it will be more convenient for our purposes to employ 
the following normalized version of the Bernstein-Szeg\"o polynomials:
\begin{equation}\label{bsp}
\texttt{p}_l(\xi) :=  \Delta_l^{-1} p_l(\xi)\quad\text{with}\quad   \Delta_l  := \int_0^\pi  p^2_l(\xi)\,  w(\xi)\text{d}\xi .
\end{equation}

A crucial observation going back to Bernstein and Szeg\"o---cf. 
\cite[Section 2.6]{sze:orthogonal}---now states that 
for $\boxed{{\textstyle l\geq d_\epsilon:= \frac{1}{2}(d-\epsilon_+ - \epsilon_-)}}$ one has explicitly:
\begin{subequations}
\begin{equation}\label{bs-explicit}
\texttt{p}_l(\xi ) =  {c}(\xi)e^{il \xi} + {c}(-\xi)e^{-il\xi} 
\end{equation}
 and (consequently)
\begin{equation}\label{bs-weights}
\Delta_l  =  
 \begin{cases} {\bigl(1 + (-1)^{\epsilon_-} \prod_{1\leq r\leq d}\alpha_r\bigr)^{-1}}  &\text{if}\  l={ d_\epsilon} ,\\
 1 &\text{if}\ l > {d_\epsilon}.
 \end{cases} 
 \end{equation}
\end{subequations}

\subsection{Composite Bernstein-Szeg\"o basis}
Let us fix two families of Bernstein-Szeg\"o polynomials $\texttt{p}_l(\xi)$ and $\tilde{\texttt{p}}_l(\xi)$ associated with the parameters
$\epsilon_\pm$, $\alpha_r$ ($r=1,\ldots ,d$) and $\tilde{\epsilon}_\pm$, $\tilde{\alpha}_r$ ($r=1,\ldots ,\tilde{d}$), respectively (subject to the domain restrictions specified above).
For positive $m$ such that
\begin{equation}\label{m-condition}
\boxed{m> \lceil d_\epsilon \rceil + \lceil \tilde{d}_{\tilde{\epsilon}}\rceil } 
\end{equation}
and $\hat{l}\in \{ 0,\ldots ,m\}$, let $\xi_{\hat{l}}^{(m)}$  be defined as the unique real solution of the transcendental equation
\begin{subequations}
\begin{equation}\label{bethe:eq}
2\bigl(m-d_\epsilon-\tilde{d}_{\tilde{\epsilon}}\bigr)\xi + \sum_{1\leq r\leq d}  \int_0^\xi u_{\alpha_r}(x)\text{d}x +    
 \sum_{1\leq r\leq \tilde{d}}  \int_0^\xi u_{\tilde{\alpha}_r}(x)\text{d}x  =\pi (2\hat{l}+\epsilon_- +\tilde{\epsilon}_-) ,
\end{equation}
where for $x\in\mathbb{R}$:
\begin{align}\label{ua}
u_\alpha(x) :=&   \frac{1-\alpha^2}{1+2\alpha\cos (x)+\alpha^2}  \qquad   (  |\alpha |<1)  \\
=& 1+2\sum_{l>0}  (-\alpha )^l  \cos (l x) . \nonumber
\end{align}
\end{subequations}
Indeed---since $\int_0^\pi u_\alpha(x)\text{d}x=\pi$ and the LHS of Eq. \eqref{bethe:eq} 
constitutes a (smooth) strictly  increasing  function of $\xi$---it is clear from this transcendental equation (via the mean value theorem) that
\begin{equation}\label{nodes}
0\leq \xi_0^{(m)}< \xi_1^{(m)}<\cdots <\xi_m^{(m)}\leq \pi.
\end{equation}
Notice that the equality $\xi^{(m)}_0=0$ is reached iff $\epsilon_-=\tilde{\epsilon}_-=0$, and that the equality $\xi^{(m)}_m=\pi$ is reached iff $\epsilon_+=\tilde{\epsilon}_+=0$.
The nodes therefore never hit a pole of
 $c(\pm \xi)$ or $\tilde{c}(\pm \xi)$, even in the situations that
the extremal values $\xi^{(m)}_0=0$ or $\xi^{(m)}_m=\pi$ are reached.

We are now in the position to introduce the following \emph{composite Bernstein-Szeg\"o basis} $\psi^{(m)}_0,\ldots,\psi^{(m)}_m$
on the nodes \eqref{nodes}:
\begin{equation}\label{bs-basis}
\psi^{(m)}_l (\xi )  :=
\begin{cases}
e^{\frac{i}{2} m\xi}  \frac{\texttt{p}_l(\xi)}{c(-\xi)} &\text{if}\  0\leq l<  d_\epsilon  ,\\
e^{\frac{i}{2} m\xi}  \frac{\texttt{p}_l(\xi)}{c(-\xi)} =
e^{-\frac{i}{2} m\xi}  \frac{\tilde{\texttt{p}}_{m-l}(\xi)}{\tilde{c}(\xi)} &\text{if}\  d_\epsilon\leq l\leq m-  \tilde{d}_{\tilde{\epsilon}} , \\
e^{-\frac{i}{2} m\xi}  \frac{\tilde{\texttt{p}}_{m-l}(\xi)}{\tilde{c}(\xi)} &\text{if}\   m-  \tilde{d}_{\tilde{\epsilon}}<l\leq m ,
\end{cases}
\end{equation}
where $\xi\in \{  \xi^{(m)}_0,\ldots ,\xi^{(m)}_m\}$ and $l\in \{ 0,\ldots ,m\}$. To avoid possible confusion with customary notation for the higher order derivatives of functions, let us emphasize at this point that throughout the presentation the superscript $(m)$ merely reflects the dependence of our construction on the number of nodes.

\subsection{Finite orthogonality relations}
For $l,\hat{l}\in \{  0,\ldots,m\}$, we define the (positive) weights
\begin{subequations}
\begin{equation}\label{pweights:a}
 {\Delta}^{(m)}_{l}:=
 \begin{cases}
  \Delta_{l}&\text{if}\ 0\leq l\leq  d_\epsilon  ,\\
  1 &\text{if} \ d_\epsilon< l< m-  \tilde{d}_{\tilde{\epsilon}} , \\
  \tilde{\Delta}_{m-l}&\text{if}\  m-  \tilde{d}_{\tilde{\epsilon}}\leq l\leq m ,
 \end{cases}
\end{equation}
and the dual (positive) weights
\begin{align}\label{pweights:b}
 \hat{\Delta}^{(m)}_{\hat{l}} 
 & :={\textstyle \left( \frac{1}{2}\right) ^{(1-\epsilon_-)(1-\tilde{\epsilon}_-)\delta_{\hat{l}}+(1-\epsilon_+)(1-\tilde{\epsilon}_+)\delta_{m-\hat{l}}}}  \times\\
& 
 \Biggl(2\bigl(m-d_\epsilon-\tilde{d}_{\tilde{\epsilon}}\bigr) + 
\sum_{1\leq r\leq d} u_{\alpha_r}\bigl(\xi^{(m)}_{\hat{l}}\bigr) +    
 \sum_{1\leq r\leq \tilde{d}}  u_{\tilde{\alpha}_r}\bigl(\xi^{(m)}_{\hat{l}} \bigr) \Biggr)^{-1} .  \nonumber
\end{align}
\end{subequations}
Here $\delta_x:=1$ if $x=0$ and $\delta_x:=0$  otherwise.

The following two (equivalent) orthogonality relations satisfied by the composite Bernstein-Szeg\"o basis constitute our principal result.
\begin{theorem}[Orthogonality Relations]\label{orthogonality:thm}
For positive $m > \lceil d_\epsilon \rceil + \lceil \tilde{d}_{\tilde{\epsilon}}\rceil $, the composite Bernstein-Szeg\"o basis \eqref{bs-basis}
satisfies the  orthogonality relations 
\begin{subequations}
\begin{equation}\label{ort:a}
\sum_{0\leq \hat{l }\leq m}   \psi^{(m)}_l  \bigl(\xi^{(m)}_{\hat{l}}\bigr)\,  \overline{\psi^{(m)}_k  \bigl(\xi^{(m)}_{\hat{l}}\bigr) } \hat{\Delta}^{(m)}_{\hat{l}}= 
\begin{cases}
1/\Delta_l^{(m)}&\text{if}\ k=l\\
0 &\text{if}\ k\neq l
\end{cases}
\end{equation}
($0\leq l,k\leq m$), or equivalently (by column-row duality)
\begin{equation}\label{ort:b}
\sum_{0\leq l  \leq m}   \psi^{(m)}_l  \bigl(\xi^{(m)}_{\hat{l}}\bigr)\,  \overline{\psi^{(m)}_l  \bigl(\xi^{(m)}_{\hat{k}}\bigr) } {\Delta}^{(m)}_{l}= 
\begin{cases}
1/\hat{\Delta}_{\hat{l}}^{(m)}&\text{if}\ \hat{k}=\hat{l}\\
0 &\text{if}\ \hat{k}\neq \hat{l}
\end{cases}
\end{equation}
\end{subequations}
($0\leq \hat{l},\hat{k}\leq m$).
\end{theorem}

For $d=\tilde{d}=2$ and $\epsilon_\pm=\tilde{\epsilon}_\pm=1$, this orthogonality can be found in \cite[Section 2.3]{die-ems:discrete}, while the degenerate case $\tilde{d}=0$ and $\tilde{\epsilon}_\pm=1$ is immediate from the Gauss quadrature rule associated with the Bernstein-Szeg\"o polynomials \cite[Section 8]{die-ems:exact}.

\section{Proof of Theorem \ref{orthogonality:thm}}\label{sec3}

\subsection{Preparatives}
Before proving the main theorem, let us first corroborate that (the gluing in) the definition of $\psi_l^{(m)}(\xi)$ \eqref{bs-basis} is legitimate, i.e. that at $\xi=\xi^{(m)}_{\hat{l}}$,
$\hat{l}\in \{ 0,\ldots , m\}$:
\begin{subequations}
\begin{equation}\label{gluing}
e^{\frac{1}{2} im\xi}  \frac{\texttt{p}_l(\xi)}{c(-\xi)} =
e^{-\frac{1}{2} im\xi}  \frac{\tilde{\texttt{p}}_{m-l}(\xi)}{\tilde{c}(\xi)} \quad \text{for}\quad d_\epsilon\leq l\leq m-  \tilde{d}_{\tilde{\epsilon}} .
\end{equation}
Recalling Eq. \eqref{bs-explicit}, one readily infers that Eq. \eqref{gluing} is satisfied when
\begin{equation}\label{bae-cf}
e^{2im\xi}=\frac{c(-\xi) \tilde{c}(-\xi)}{c(\xi)\tilde{c}(\xi)} ,
\end{equation}
or more explicitly (by Eq. \eqref{c-f}):
\begin{equation}\label{bae}
e^{2im\xi}=(-1)^{\epsilon_- +\tilde{\epsilon}_-} e^{2i(d_\epsilon+\tilde{d}_{\tilde{\epsilon}})\xi} 
\prod_{1\leq r\leq d}  \frac{1+\alpha_re^{i\xi}}{e^{i\xi}+\alpha_r}
\prod_{1\leq r\leq \tilde{d}}  \frac{1+\tilde{\alpha}_re^{i\xi}}{e^{i\xi}+\tilde{\alpha}_r} .
\end{equation}
\end{subequations}
That Eq. \eqref{bae} holds at $\xi=\xi^{(m)}_{\hat{l}}$ is immediate from Eqs. \eqref{bethe:eq}, \eqref{ua}, upon multiplying this defining transcendental equation for  $\xi^{(m)}_{\hat{l}}$  by the imaginary unit and exponentiating both sides with the aid of the identity
\begin{equation}\label{identity}
\exp \left( - i\int_0^\xi u_\alpha(x)\text{d}x \right)=  \frac{1+\alpha e^{i\xi}}{e^{i\xi}+\alpha}\qquad (|\alpha |<1).
\end{equation}

Now turning to the proof of the theorem, we observe that both types of orthogonality relations formulated in Theorem \ref{orthogonality:thm}  amount to the claim that the $(m+1)$-dimensional matrix
\begin{equation*}
\left[ \sqrt{\Delta^{(m)}_l \hat{\Delta}^{(m)}_{\hat{l}}} \psi^{(m)}_l \bigl(\xi^{(m)}_{\hat{l}} \bigr) \right]_{0\leq l,\hat{l}\leq m}
\end{equation*}
 is unitary. It is therefore sufficient to verify either one of them, and here we choose to infer Eq. \eqref{ort:b}. Rewritten explicitly in terms of Bernstein-Szeg\"o polynomials, this orthogonality relation states that for any $\ell \in \{ 0,\ldots ,m\} $ such that  $\lceil  \tilde d_{\tilde \epsilon}\rceil < \ell\leq m- \lceil d_\epsilon\rceil $ (which exists because
 $m >  \lceil d_\epsilon \rceil +\lceil \tilde d_{\tilde \epsilon }\rceil $):
 \begin{align}\label{split}
& \frac{1}{c\bigl(-\xi^{(m)}_{\hat{l}}\bigr)c\bigl(\xi^{(m)}_{\hat{k}}\bigr)}\sum_{l=0}^{m-\ell}  \texttt{p}_l  \bigl(\xi^{(m)}_{\hat{l}}\bigr) \texttt{p}_l \bigl(\xi^{(m)}_{\hat{k}}\bigr) \Delta_l  \\
 &+ \frac{1}{\tilde c\bigl(\xi^{(m)}_{\hat{l}}\bigr) \tilde c\bigl(-\xi^{(m)}_{\hat{k}}\bigr)}\sum_{l=0}^{\ell-1}\tilde{\texttt{p}}_l \bigl(\xi^{(m)}_{\hat{l}}\bigr) \tilde{\texttt{p}}_l \bigl(\xi^{(m)}_{\hat{k}}\bigr)  \tilde{\Delta}_l 
=\begin{cases}
1/\hat{\Delta}_{\hat{l}}^{(m)}&\text{if}\ \hat{k}=\hat{l}\\
0 &\text{if}\ \hat{k}\neq \hat{l}
\end{cases}  \nonumber
\end{align}
($0\leq \hat{l},\hat{k}\leq m$).

\subsection{Orthogonality}
In the normalization \eqref{bsp}, the 
three-term recurrence relation for the Bernstein-Szeg\"o polynomials takes the form
\begin{subequations}
\begin{equation}\label{recurrence}
2\cos (\xi) \texttt{p}_l (\xi) =   \texttt{p}_{l-1} (\xi)+  b_l  \texttt{p}_l  (\xi) + a_{l+1}^2 \text{p}_{l+1}(\xi)\quad \text{with}\  a_{l+1}:= \left( \frac{\Delta_{l+1}}{\Delta_l}\right)^{1/2},
\end{equation}
for certain (real) coefficients $b_l$, $l=0,1,2,\ldots$ (and $\texttt{p}_{-1}(\xi):=0$). From the explicit formulas \eqref{bs-explicit}, \eqref{bs-weights} it is moreover seen that
\begin{equation}\label{stabilize}
a_{l+1}=1 \ \text{for}\ l >  d_\epsilon \quad  \text{and} \quad  b_l=0\ \text{for}\  l > \lceil d_\epsilon \rceil  .
\end{equation}
\end{subequations}
With the aid of the corresponding three-term recurrences for $ \texttt{p}_l  (\xi) $ and  $\tilde{ \texttt{p}}_l  (\xi) $, we now construct a real tridiagonal matrix
that is diagonalized by the composite Bernstein-Szeg\"o basis \eqref{bs-basis}:
\begin{subequations}
\begin{equation}\label{Lm:a}
L^{(m)}:=\left[  L^{(m)}_{l,k} \right]_{0\leq l,k\leq m}
\end{equation}
with
\begin{align}\label{Lm:b}
&L^{(m)}_{l,k} := \\
&\begin{cases}
 \delta_{l-k-1}+ b_l\delta_{l-k} +a_{l+1}^2 \delta_{l-k+1}, &\text{if}\  0\leq l \leq \lceil d_\epsilon \rceil  ,\\
 \delta_{l-k-1}+ \delta_{l-k+1}&\text{if}\   \lceil d_\epsilon\rceil <  l <  m- \lceil \tilde{d}_{\tilde{\epsilon}} \rceil ,\\
 \tilde{a}_{m-l+1}^2 \delta_{l-k-1}+ \tilde{b}_{m-l} \delta_{l-k} + \delta_{l-k+1}, &\text{if}\  m-\lceil \tilde{d}_{\tilde{\epsilon}} \rceil \leq l\leq m .
\end{cases}
\nonumber
\end{align}
\end{subequations}
Upon interpreting  functions $f:\{ 0,\ldots ,m\}  \to \mathbb{C}$ as column vectors $[ f_l ]_{0\leq l\leq m}$, the action of
$L^{(m)}$ \eqref{Lm:a}, \eqref{Lm:b} on $f$ becomes:
\begin{equation}\label{Lm-action}
 (L^{(m)} f)_l := 
 \begin{cases}
 f_{l-1} + b_l f_l+a_{l+1}^2 f_{l+1}, &\text{if}\  0\leq l  \leq \lceil d_\epsilon \rceil  ,\\
 f_{l-1}+f_{l+1} , &\text{if}\   \lceil d_\epsilon\rceil <  l < m - \lceil \tilde{d}_{\tilde{\epsilon}} \rceil ,\\
\tilde{a}_{m-l+1}^2  f_{l-1} + \tilde{b}_{m-l} f_l+ f_{l+1} &\text{if}\  m-\lceil \tilde{d}_{\tilde{\epsilon}}\rceil  \leq  l\leq m  
\end{cases}
\end{equation}
(where $f_{-1}=f_{m+1}:=0$).

\begin{proposition}[Diagonalization of $L^{(m)}$]\label{diagonalization:prp}
Viewed as a function of $l\in \{ 0,\ldots ,m\}$, the composite Bernstein-Szeg\"o
vector $\psi^{(m)}_l (\xi)$ \eqref{bs-basis} solves the eigenvalue equation
\begin{equation}\label{ev-eq}
L^{(m)}\psi^{(m)}(\xi) = 2\cos(\xi ) \psi^{(m)} (\xi) ,
\end{equation}
for any $\xi\in \{ \xi_0^{(m)},\ldots ,\xi_m^{(m)}\} $.
\end{proposition}
\begin{proof}
If $0\leq l < m-\lceil \tilde{d}_{\tilde{\epsilon}}\rceil $ then
\begin{equation*}
\bigl( L^{(m)} \psi^{(m)} (\xi)\bigr)_l \stackrel{ i)}{= }   \psi ^{(m)}_{l-1} (\xi)+ b_l  \psi^{(m)}_l  (\xi)  +a_{l+1}^2  \psi^{(m)} _{l-1} (\xi)
\stackrel{ii) }{=} 2\cos (\xi)   \psi^{(m)}_l  (\xi)  ,
\end{equation*}
and if $\lceil d_\epsilon\rceil <  l \leq  m$ then
\begin{equation*}
\bigl(L^{(m)} \psi^{(m)} (\xi)\bigr)_l \stackrel{i)}{=} 
\tilde{a}_{m-l+1}^2   \psi^{(m)} _{l-1} (\xi) + \tilde{b}_{m-l} \psi ^{(m)}_{l} (\xi)+  \psi^{(m)} _{l+1} (\xi)
\stackrel{ii)}{=} 2\cos (\xi)   \psi^{(m)}_l  (\xi)  .
\end{equation*}
Here we used $i)$ Eqs.~\eqref{Lm-action}, \eqref{stabilize} and $ ii) $ Eqs.~\eqref{bs-basis}, \eqref{recurrence}.
\end{proof}

Since  the (eigen)vectors
$\psi^{(m)}\bigl(\xi^{(m)}_{0}\bigr),\ldots ,\psi^{(m)}\bigl(\xi^{(m)}_{m}\bigr)$ are nonzero (because for any $0\leq \hat{l}\leq m$: $\psi^{(m)}_0\bigl(\xi^{(m)}_{\hat{l}}\bigr)\neq 0$, as  $\texttt{p}_0 (\xi)=\Delta_0^{-1}\neq 0$),
and the eigenvalues $2\cos \bigl(\xi^{(m)}_0\bigr),\ldots ,2\cos \bigl(\xi^{(m)}_m\bigr)$ are simple in view of Eq. \eqref{nodes},
it is evident that Proposition \ref{diagonalization:prp} indeed provides an eigenbasis for $L^{(m)}$ \eqref{Lm:a}, \eqref{Lm:b}.

Let $\Delta^{(m)}$ denote the positive $(m+1)$-dimensional diagonal matrix
\begin{equation}\label{Deltam}
\Delta^{(m)}:= \text{diag} \bigl( \Delta^{(m)}_0,  \Delta^{(m)}_1,\ldots , \Delta^{(m)}_m\bigr) .
 \end{equation}
 The orthogonality in  Eq. \eqref{ort:b} asserts that the vectors
 \begin{equation*}
\bigl( \Delta^{(m)}\bigr)^{1/2} \psi^{(m)}\bigl(\xi^{(m)}_{0}\bigr),\ldots ,\bigl( \Delta^{(m)}\bigr)^{1/2} \psi^{(m)}\bigl(\xi^{(m)}_{m}\bigr)
 \end{equation*}
 are orthogonal. In view of the diagonalization stemming from Proposition \ref{diagonalization:prp}, this orthogonality is clear after corroborating that the
 similarity transformation
 \begin{equation}\label{Jm}
 J^{(m)}:=\bigl( \Delta^{(m)}\bigr)^{1/2}   L^{(m)} \bigl( \Delta^{(m)}\bigr)^{-1/2}
 \end{equation}
  yields a Jacobi matrix.

\begin{proposition}[Jacobi Matrix $J^{(m)}$]\label{jacobi:prp}
The (real) tridiagonal matrix  $J^{(m)}$ \eqref{Jm} is symmetric, with elements $a_1^{(m)},\ldots ,a_m^{(m)}$ on the sub- and superdiagonals given by
\begin{subequations}
\begin{equation}
a_{l+1}^{(m)}:= \begin{cases}
a_{l+1} & \text{if}\  \  0\leq l < \lceil d_\epsilon\rceil  ,\\
\bigl( \Delta_l \tilde{\Delta}_{m-l-1}\bigr)^{-1/2} &\text{if}\  \ \lceil d_\epsilon\rceil \leq  l < m - \lceil \tilde{d}_{\tilde{\epsilon}} \rceil  , \\
\tilde{a}_{m-l}&\text{if}\   \  m-\lceil \tilde{d}_{\tilde{\epsilon}} \rceil \leq l < m  ,
\end{cases}
\end{equation}
and elements  $b_0^{(m)},\ldots ,b_m^{(m)}$ on the principal diagonal of the form
\begin{equation}
b_l^{(m)}:= \begin{cases}
b_l &\text{if}\  \  0\leq l \leq \lceil d_\epsilon\rceil  ,\\
0 &\text{if}\  \ \lceil d_\epsilon\rceil <  l < m - \lceil \tilde{d}_{\tilde{\epsilon}} \rceil  , \\
\tilde{b}_{m-l}&\text{if}\   \  m-\lceil \tilde{d}_{\tilde{\epsilon}} \rceil \leq l\leq m .
\end{cases}
\end{equation}
\end{subequations}
\end{proposition}
\begin{proof}
Starting from the formulas for $L^{(m)}$  and $\Delta^{(m)}$ in Eqs. \eqref{Lm:a}, \eqref{Lm:b} and Eq.  \eqref{Deltam}, respectively, one readily verifies
(upon recalling the definitions in Eqs.
\eqref{pweights:a} and
\eqref{recurrence}) that  
$J^{(m)}_{l,l+1}=\bigl( \Delta_l^{(m)}\bigr)^{1/2} L^{(m)}_{l,l+1} \bigl(\Delta^{(m)}_{l+1}\bigr)^{-1/2}=a^{(m)}_{l+1}$
and that
$J^{(m)}_{l+1,l}=\bigl( \Delta_{l+1}^{(m)}\bigr)^{1/2} L^{(m)}_{l+1,l}
\bigl(\Delta^{(m)}_{l}\bigr)^{-1/2} =a^{(m)}_{l+1}$   ($l=0,\ldots,m-1$), whereas on the diagonal
it is obvious that $J^{(m)}_{l,l}=L^{(m)}_{l,l}=b_l^{(m)}$ ($l=0,\ldots ,m$).
\end{proof}

This completes the proof of Eq. \eqref{ort:b}  when $\hat{k}\neq \hat{l}$.

\subsection{Quadratic norms}
It remains to infer the quadratic norm formula in  Eq. \eqref{ort:b} when $\hat{k}= \hat{l}$.
This amounts to the following summation formula.

\begin{proposition}[Summation Formula]
\label{prop:dual-weights}
For any $\xi\in\{ \xi_0^{(m)},\ldots ,\xi^{(m)}_m\}$, one has that
\begin{align}
 \sum_{0\leq l\leq m}  & |\psi_l^{(m)} (\xi )|^2 \Delta^{(m)}_l
= \\
&2^{\delta_\xi+\delta_{\pi-\xi}}\left(
2\bigl(m-d_\epsilon - \tilde{d}_{\tilde{\epsilon}}\bigr)+\sum_{1\leq r\leq d} u_{\alpha_r}(\xi )+\sum_{1\leq r\leq \tilde d} u_{\tilde{\alpha}_r}(\xi ) \right) . \nonumber
\end{align}

\end{proposition}

\begin{proof}
For $\xi\in (0,\pi)$ the proof hinges on the Christoffel-Darboux kernel  for the Bernstein-Szeg\"o polynomials \cite{sze:orthogonal,gri:christoffel} evaluated on the diagonal:
\begin{equation*}
\sum_{0\leq l\leq N} \texttt{p}_l^2 (\xi) \Delta_l =  
 \frac{\Delta_{N+1}}{2\sin(\xi) }  \Bigl(  \texttt{p}_{N+1}(\xi)\texttt{p}_N^\prime(\xi) - \texttt{p}_{N+1}^\prime (\xi) \texttt{p}_N(\xi)  \Bigl) .
\end{equation*}
For $N\ge d_\epsilon$, the RHS can be computed explicitly by means of Eqs. \eqref{bs-explicit}, \eqref{bs-weights}:
\begin{align*}
\sum_{0\leq l\leq N} \texttt{p}_l^2 (\xi) \Delta_l =  &
 c(\xi)c(-\xi)
\left[
\frac{1}{i} \left( \frac{c'(\xi)}{c(\xi)}  +\frac{c'(-\xi)}{c(-\xi)} \right)
 + 2N+1 
 \right] 
 \\
&+  
(2i\sin \xi)^{-1}
\left(
 e^{i(2N+1)\xi}  c^2(\xi)
- e^{-i(2N+1)\xi}  c^2(-\xi)
 \right)  .
\end{align*}

By  using the definition of the composite Bernstein-Szeg\"o basis in Eq. \eqref{bs-basis} and splitting the sum as in Eq. \eqref{split},
we get (assuming $\xi\in\{ \xi_0^{(m)},\ldots ,\xi^{(m)}_m\}$):
\begin{equation*}
 \sum_{0\leq l\leq m}|\psi_l^{(m)} (\xi )|^2 \Delta^{(m)}_l
 =
 \frac{1}{c(\xi)c(-\xi)}\sum_{l=0}^{m-\ell}  \texttt{p}_l^2 (\xi) \Delta_l
+ \frac{1}{\tilde c(\xi) \tilde c(-\xi)}\sum_{l=0}^{\ell-1}\tilde{\texttt{p}}_l^2(\xi) \tilde{\Delta}_l ,
\end{equation*}
for
$\ell \in \{ 0,\ldots ,m\} $ such that  $\lceil  \tilde d_{\tilde \epsilon}\rceil < \ell \leq m- \lceil d_\epsilon\rceil $.
Summation of both parts with the aid of the explicit formula for the diagonal of the Christoffel-Darboux kernel yields:
\begin{align*}
 \frac{1}{\tilde c(\xi) \tilde c(-\xi)}\sum_{l=0}^{\ell-1}  \tilde{\texttt{p}}_l^2 (\xi)\tilde{\Delta}_l
=&
\left[
\frac{1}{i} \left( \frac{\tilde c'(\xi)}{\tilde{c}(\xi)}  +\frac{\tilde c'(-\xi)}{\tilde{c}(-\xi)} \right)
 + 2\ell-1 
 \right] 
 \\
& 
+
(2i\sin \xi)^{-1}
\left((
 e^{i(2\ell-1)\xi}  \frac{\tilde{c}(\xi)}{\tilde{c}(-\xi)}
- e^{-i(2\ell -1)\xi}  \frac{\tilde{c}(-\xi)}{\tilde{c}(\xi)}
 \right)
\end{align*}
and
\begin{align*}
\frac{1}{c(\xi)c(-\xi)}\sum_{l=0}^{m-\ell}  \texttt{p}_l^2 (\xi)\Delta_l
 =&
\left[
\frac{1}{i} \left(\frac{c'(\xi)}{c(\xi)}  +\frac{c'(-\xi)}{c(-\xi)} \right)
 + 2(m-\ell )+1 
 \right] 
 \\
& 
+(2i\sin \xi)^{-1}
\left(
 e^{i(2(m-\ell )+1)\xi}  \frac{c(\xi)}{c(-\xi)}
- e^{-i(2(m-\ell )+1)\xi}  \frac{c(-\xi)}{c(\xi)} \right)\\
=&
\left[
\frac{1}{i} \left(\frac{c'(\xi)}{c(\xi)}  +\frac{c'(-\xi)}{c(-\xi)} \right)
 + 2(m-\ell )+1 
 \right] 
 \\
& 
+ (2i\sin \xi)^{-1}
\left(
 e^{-i(2\ell-1)\xi}  \frac{\tilde{c}(-\xi)}{\tilde{c}(\xi)}
- e^{i(2\ell -1)\xi}  \frac{\tilde{c}(\xi)}{\tilde{c}(-\xi)} \right),
\end{align*}
where in the last step  Eq. \eqref{bae-cf} was used. Pasting everything together  finally entails
the desired result:
\begin{equation*}
 \sum_{0\leq l\leq m}|\psi_l^{(m)} (\xi )|^2 \Delta^{(m)}_l=2m+
\frac{1}{i} \left( \frac{c'(\xi)}{c(\xi)}  +\frac{c'(-\xi)}{c(-\xi)} 
+ \frac{\tilde c'(\xi)}{\tilde{c}(\xi)}  +\frac{\tilde c'(-\xi)}{\tilde{c}(-\xi)} \right) 
\end{equation*} 
(which simplifies to the stated norm formula upon inserting  the explicit expressions for $c(\xi)$ and $\tilde{c}(\xi)$ stemming from Eq. \eqref{c-f} into the
logarithmic derivatives). When $\xi=0$ or $\xi=\pi$, the summation formula follows similarly upon applying
L'H\^{o}pital's rule to the expression for the diagonal of  the Christoffel-Darboux kernel at the start (which entails an extra factor $2$).
\end{proof}

This completes the proof of Eq. \eqref{ort:b}  when $\hat{k}= \hat{l}$.

\section{Miscellanea}\label{sec4}

\subsection{Characteristic polynomial}
It is instructive to observe that the spectral values providing our nodes   $0\leq \xi_0^{(m)}<\cdots <\xi^{(m)}_m\leq \pi$ can be retrieved as the roots of the following polynomial of degree $m+1$
in $\cos (\xi)$:
\begin{subequations}
\begin{equation}\label{node-pol:a}
{Q}_{m+1}(\xi) := \sum_{0\leq k\leq 2( \tilde{d}_{\tilde{\epsilon}}+1)}  \text{e}_k(\tilde{\alpha}; 1-\tilde{\epsilon}_+;1-\tilde{\epsilon}_- )\, \texttt{q}_{m+1-k}(\xi)  ,
\end{equation}
where $ \texttt{q}_l(\xi):=c(\xi) e^{il\xi} + c(-\xi) e^{-il\xi} $ for $l\in\mathbb{Z}$ (so $\texttt{q}_l(\xi)=\texttt{p}_l(\xi)$ if $l\geq d_\epsilon$) and
\begin{equation*}
e_k(\tilde{\alpha};\tilde{\epsilon}_+; \tilde{\epsilon}_-) :=\text{e}_k(\tilde{\alpha})+(\tilde{\epsilon}_+ - \tilde{\epsilon}_-)\text{e}_{k-1}(\tilde{\alpha})- \tilde{\epsilon}_+\tilde{\epsilon}_-\text{e}_{k-2}(\tilde{\alpha}) ,
\end{equation*}
with
\begin{equation*}
\text{e}_k(\tilde{\alpha}) :=
\sum_{1\leq r_1<\cdots <r_k\leq\tilde{d}} \tilde{\alpha}_{r_1} \cdots  \tilde{\alpha}_{r_k} 
\end{equation*}
subject to the conventions that $\text{e}_0(\tilde{\alpha}):=1$ and $\text{e}_{k}(\tilde{\alpha}):=0$ if $k\not\in \{ 0,1,\ldots ,\tilde{d}\}$.

More precisely, one has that
\begin{equation}\label{node-pol:b}
Q_{m+1}(\xi) = 2^{m+1} \left( \cos(\xi)-\cos\bigl(  \xi^{(m)}_0  \bigr) \right)  \cdots  \left( \cos(\xi)-\cos\bigl(  \xi^{(m)}_m  \bigr) \right) .
\end{equation}
\end{subequations}
Indeed, by means of the explicit formula for $\texttt{q}_l(\xi)$ it is readily seen that the equation $Q_{m+1}(\xi)=0$ boils down to Eqs. \eqref{bae-cf} and \eqref{bae}. To this end $Q_{m+1}(\xi)$ \eqref{node-pol:a} is first split in two parts, with common factors $c(\xi)e^{i(m+1)\xi}$ and $c(-\xi)e^{-i(m+1)\xi}$, respectively, and then  the terms of each factor are collected with the aid of
the generating function
\begin{equation*}
(1+\tilde{\epsilon}_+z)(1-\tilde{\epsilon}_-z)\prod_{1\leq r \leq \tilde{d}} (1+\tilde{\alpha}_r z)=
\sum_{l=0}^{ \tilde{d}+\tilde{\epsilon}_++\tilde{\epsilon}_-}  \text{e}_k (\tilde{\alpha};\tilde{\epsilon}_+; \tilde{\epsilon}_-) z^k
\end{equation*}
(at $z=e^{\mp i\xi}$).
To confirm that the degree of $Q_{m+1}(\xi)$ in $\cos (\xi )$ is $m+1$, one uses that
\begin{equation*}
\deg( \texttt{q}_l) =\begin{cases} l&\text{if}\  l\geq  d_\epsilon , \\
2d_\epsilon - l &\text{if}\  l< d_\epsilon,
\end{cases}
\end{equation*}
in combination with our condition on $m$ in Eq. \eqref{m-condition},
whereas to compute the coefficient of the corresponding leading term (stemming from $\texttt{q}_{m+1}(\xi)=\texttt{p}_{m+1}(\xi)$) one employs
Eqs. \eqref{bsp} and \eqref{bs-explicit}, \eqref{bs-weights} (thus verifying the factorization of $Q_{m+1}(\xi)$ in  Eq. \eqref{node-pol:b}).

If we compare Eq. \eqref{node-pol:b} with Eq. \eqref{ev-eq},  then it is clear that $Q_{m+1}(\xi)$ \eqref{node-pol:a} coincides with the
characteristic polynomials of the $(m+1)$-dimensional tridiagonal matrices $L^{(m)}$ and $J^{(m)}$ from Propositions \ref{diagonalization:prp} and \ref{jacobi:prp}, respectively:
\begin{equation}
\det  \left( 2\cos (\xi)I^{(m)}- L^{(m)}\right)=\det  \left( 2\cos (\xi)I^{(m)}- J^{(m)}\right)=Q_{m+1}(\xi)
\end{equation}
(where $I^{(m)}$ refers to the $(m+1)$-dimensional identity matrix).

\subsection{Estimates for the locations of the nodes}
With the aid of the mean value theorem
it is immediate from the transcendental equation for $\xi^{(m)}_{\hat{l}}$ in Eqs. \eqref{bethe:eq}, \eqref{ua},  that
the nodes/spectral values $\xi_0^{(m)},\ldots ,\xi_m^{(m)}$ \eqref{nodes} satisfy the following inequalities:
\begin{subequations}
\begin{equation}\label{bound:a}
\frac{\pi \bigl(\hat{l}+\frac{\epsilon_- +\tilde{\epsilon}_-}{2} \bigr)}{m -{d_\epsilon}- \tilde{d}_{\tilde{\epsilon}}+\kappa_-}   \leq \xi_{\hat{l}}^{(m)}  \leq 
\frac{\pi \bigl(\hat{l}+\frac{\epsilon_- +\tilde{\epsilon}_-}{2} \bigr)}{m -{d_\epsilon}-\tilde{d}_{\tilde{\epsilon}}+\kappa_+} 
\end{equation}
(for $0\leq \hat{l}\leq m$), and
\begin{equation}\label{bound:b}
\frac{\pi (\hat{k}-\hat{l})}{m -{d_\epsilon}- \tilde{d}_{\tilde{\epsilon}}+\kappa_-}   \leq \xi_{\hat{k}}^{(m)} - \xi_{\hat{l}}^{(m)} \leq 
\frac{\pi (\hat{k}-\hat{l})}{m -{d_\epsilon}- \tilde{d}_{\tilde{\epsilon}}+\kappa_+} 
\end{equation}
(for $0\leq \hat{l}<\hat{k}\leq m$), where
\begin{equation}\label{kappa}
\kappa_\pm := \frac{1}{2}\sum_{1\leq r\leq d}   \left(\frac{1- |\alpha_r|}{1 + |\alpha_r|}\right)^{\pm 1}   + 
  \frac{1}{2}\sum_{1\leq r\leq \tilde{d}}   \left(\frac{1- |\tilde{\alpha}_{r|}}{1 + |\tilde{\alpha}_{r|}}\right)^{\pm 1}  .
\end{equation}
\end{subequations}
Here one uses that for $\xi$ real
\begin{equation*}
 \text{Re} \bigl(u_\alpha (\xi)\bigr) =  \frac{1}{2}  \Bigl(    u_{|\alpha |} \bigl(\xi+\text{Arg}(\alpha)\bigr)+ u_{|\alpha |} \bigl(\xi-\text{Arg}(\alpha )\bigr) \Bigr) ,
\end{equation*}
and thus
\begin{equation*}
\frac{1-|\alpha |}{1 + |\alpha |} \leq \text{Re} \bigl(u_\alpha (\xi)\bigr) \leq    \frac{1+|\alpha |}{1 - |\alpha |}  \qquad (|\alpha |<1).
\end{equation*}
(To recover the second estimate, for the distances between the nodes, one first subtracts the 
$\hat{l}$th equation in Eq. \eqref{bethe:eq} from the $\hat{k}$th equation before invoking the mean value theorem.)

\begin{note}\label{numerics:rem}
When all Bernstein-Szeg\"o parameters $\alpha_r$ and $\tilde{\alpha}_r$ tend to $0$,  the above estimates become exact and we reduce to the elementary situation
corresponding to $d=\tilde{d}=0$:
\begin{equation}\label{exact-node}
\xi_{\hat{l}}^{(m)}= 
\frac
{\pi (2\hat{l}+\epsilon_- +\tilde{\epsilon}_-  )}
{2m+\epsilon_+ +\epsilon_-  + \tilde{\epsilon}_+ +\tilde{\epsilon}_-} \qquad (\text{for}\ d,\tilde{d}=0).
\end{equation}
More generally the positions of the nodes  $\xi_{\hat{l}}^{(m)}$ can be effectively  retrieved
numerically from Eqs. \eqref{bethe:eq}, \eqref{ua}, e.g. by means of a standard  fixed-point iteration scheme like Newton's method.
Notice in this connection that the numerical integration of $u_\alpha (x)$ in Eq. \eqref{bethe:eq} can be avoided at this point, since (cf. Eq. \eqref{identity})
\begin{equation*}
\int_0^\xi u_\alpha(x) \text{d}x=i \text{Log}
\biggl( \frac{1+ \alpha e^{i\xi}}{e^{i\xi} +\alpha}  \biggr) =  2\text{Arctan} \Biggl(    \frac{1-\alpha}{1+\alpha}\tan \left( \frac{\xi}{2} \right)   \Biggr)
\end{equation*}
(for $|\alpha |<1$ and  $-\pi< \xi<\pi$). Convenient initial values for starting up such
numerical calculations are provided by the locations of the elementary nodes  at vanishing parameter values from Eq. \eqref{exact-node}. Indeed, this initial configuration automatically
complies with all inequalities in Eqs. \eqref{bound:a}--\eqref{kappa}.
\end{note}

\subsection{Finite orthogonality for  Askey-Wilson polynomials at $q=0$}
If $d=\tilde{d}=4$ and $\epsilon_\pm=\tilde{\epsilon}_\pm=1$ (so $d_\epsilon=\tilde{d}_{\tilde{\epsilon}}=1$), then the Bernstein-Szeg\"o weight function  coincides with that of the
Askey-Wilson polynomials \cite{ask-wil:some,koe-swa:hypergeometric} at $q=0$:
\begin{equation}
w(\xi)= \frac{1}{2\pi |c(\xi)|^2}   , \quad c(\xi)  =  \frac{\prod_{r=1}^4 (1+\alpha_re^{-i\xi})}{1-e^{-2i\xi}}   \qquad (0<\xi <\pi).
\end{equation}
We then have explicitly
\begin{subequations}
\begin{equation}
\texttt{p}_l(\xi) = \begin{cases}
\Delta_0^{-1}&\text{if}\   l=0, \\
c(\xi) e^{il\xi}+c(-\xi)e^{-il\xi}&\text{if}\ l>0,
\end{cases}
\end{equation}
and
 \begin{equation}
\Delta_l=
\begin{cases}
\frac{1-\alpha_1\alpha_2\alpha_3\alpha_4}{\prod_{1\leq r<s\leq 4} (1-\alpha_r\alpha_s)} &\text{if}\ l=0,\\
  \frac{1}{1-\alpha_1\alpha_2\alpha_3\alpha_4}&\text{if}\ l=1,\\
  1 &\text{if}\  l>1 .
  \end{cases}
\end{equation}
\end{subequations}

For any  $m> 2$, the associated composite $q=0$ Askey-Wilson vectors
\begin{subequations}
\begin{align}
&\psi^{(m)}_l (\xi)=\\
&\begin{cases}
e^{\frac{i}{2} m\xi} \frac{\Delta_0^{-1}}{c(-\xi)} &\text{if}\ l=0,\\
e^{\frac{i}{2} m\xi}\bigl(\frac{c(\xi)}{c(-\xi)} e^{il\xi}  +e^{-il\xi }\bigr)=
e^{-\frac{i}{2} m\xi}\bigl(     e^{i(m-l)\xi}  + \frac{\tilde{c}(-\xi)}{\tilde{c}(\xi)} e^{-i(m-l)\xi }\bigr)
&\text{if}\   0< l<m,\\
e^{-\frac{i}{2} m\xi}\frac{\tilde{\Delta}_0^{-1}}{\tilde{c}(\xi)} &\text{if}\ l=m,
\end{cases} \nonumber
\end{align}
at $\xi=\xi^{(m)}_{\hat{l}}$ solving
\begin{equation}
2(m-2)\xi + \sum_{1\leq r\leq 4}  \int_0^\xi \bigl( u_{\alpha_r}(x) + u_{\tilde{\alpha}_r}(x) \bigr) \text{d}x  =2\pi (\hat{l}+1) 
\end{equation}
\end{subequations}
($0\leq \hat{l}\leq m$), diagonalize
the $(m+1)$-dimensional tridiagonal matrix 
\begin{equation}
L^{(m)}=
\begin{bmatrix}
b_0 & a_1^2 &   &   &  &  &          \\
1 &b_1 & a_2^2 &   &   & &           \\
& 1& 0 &  1 &    &  &      \\
         && \ddots & \ddots&  \ddots &    &     \\
 & & &1&0&1 &     \\
   & &  &   & \tilde{a}_2^2 &\tilde{b}_1 & 1       \\
    & & & & &\tilde{a}_1^2 &\tilde{b}_0      
\end{bmatrix} ,
\end{equation}
where
\begin{equation*}
a_1=\bigl(\Delta_1/\Delta_0\bigr)^{1/2} ,\qquad a_2=\Delta_1^{-1/2}, 
\end{equation*}
\begin{equation*}
 b_0=
(\Delta_1-1) \sum_{1\leq r\leq 4} \alpha_r^{-1}-
    \Delta_1 \sum_{1\leq r\leq 4} \alpha_r, 
    \end{equation*}
    \begin{equation*}
 b_1=     (\Delta_1-1)  \sum_{1\leq r\leq 4} (\alpha_r-\alpha_r^{-1})
 \end{equation*}
(and with analog formulas for $\tilde{a}_{l+1}$, $\tilde{b}_l$ and $\tilde{\Delta}_l$, $l=0,1$).

Given (any)  $1<    \ell \leq m-1$, the $q=0$ Askey-Wilson basis satisfies the composite finite orthogonality relations (cf. Eq. \eqref{split}):
 \begin{align}
& \frac{1}{c\bigl(-\xi^{(m)}_{\hat{l}}\bigr)c\bigl(\xi^{(m)}_{\hat{k}}\bigr)}\sum_{l=0}^{m-\ell}  \texttt{p}_l  \bigl(\xi^{(m)}_{\hat{l}}\bigr) \texttt{p}_l \bigl(\xi^{(m)}_{\hat{k}}\bigr) \Delta_l \, +\\
 & \frac{1}{\tilde c\bigl(\xi^{(m)}_{\hat{l}}\bigr) \tilde c\bigl(-\xi^{(m)}_{\hat{k}}\bigr)}\sum_{l=0}^{\ell-1}\tilde{\texttt{p}}_l \bigl(\xi^{(m)}_{\hat{l}}\bigr) \tilde{\texttt{p}}_l \bigl(\xi^{(m)}_{\hat{k}}\bigr)  \tilde{\Delta}_l \, =  \nonumber \\
&\begin{cases}
2(m-2)+ 
\sum_{1\leq r\leq 4} \left( u_{\alpha_r}\bigl(\xi^{(m)}_{\hat{l}}\bigr) +    
   u_{\tilde{\alpha}_r}\bigl(\xi^{(m)}_{\hat{l}} \bigr) \right)  &\text{if}\ \hat{k}=\hat{l}\\
0 &\text{if}\ \hat{k}\neq \hat{l}
\end{cases}  \nonumber
\end{align}
($0\leq \hat{l},\hat{k}\leq m$).

\section{Exact quadratures for rational functions with prescribed poles}\label{sec5}
From the orthogonality in Eq. \eqref{ort:a}, the following quadrature formulae for the exact integration of rational functions with prescribed poles
against the Chebyshev weight functions
are immediate via standard arguments (cf. e.g. \cite{sze:orthogonal,gau:survey,dav-rab:methods}).

\begin{theorem}[Exact Quadrature Rules for Rational Functions]\label{quadrature:thm}
Let $m$ be positive such that $m>\lceil d_\epsilon \rceil +\lceil \tilde{d}_{\tilde{\epsilon}} \rceil$ (and let us assume that the parameters meet the restrictions in Section \ref{sec2}).
Then the following (positive) quadrature rule holds true:
\begin{subequations}
\begin{equation}\label{q-rule:a}
\frac{1}{2\pi} \int_0^\pi  R (\xi)  \rho  (\xi) \text{d}\xi =
 \sum_{0\leq \hat{l}\leq m}   R\bigl( \xi _{\hat{l}}^{(m)}\bigr)  \rho \bigl(\xi _{\hat{l}}^{(m)}\bigr) \hat{\Delta}^{(m)}_{\hat{l}} ,
\end{equation}
where the nodes $0\leq \xi^{(m)}_0< \xi^{(m)}_2<\cdots <\xi^{(m)}_m\leq \pi$ are determined by Eqs. \eqref{bethe:eq}, \eqref{ua} and
the Christoffel weights $\hat{\Delta}_0,\ldots ,\hat{\Delta}_m$ are given by Eq. \eqref{pweights:b}.
In this identity $\rho (\cdot )$ refers to the Chebyshev weight function
\begin{equation}\label{q-rule:b}
\rho (\xi):=2^{\epsilon_++\epsilon_-}(1+\epsilon_+\cos(\xi))   (1-\epsilon_-\cos(\xi)) ,
\end{equation}
and $R(\cdot)$ denotes  a rational function of the form
\begin{equation}\label{q-rule:c}
R(\xi )  =   \frac{f\bigl(\cos (\xi)\bigr)}{\prod_{1\leq r\leq d} \bigl(1+2\alpha_r\cos (\xi )+\alpha_r^2\bigr)}   ,
\end{equation}
where $f\bigl(\cos(\xi)\bigr)$ represents an arbitrary polynomial   in $\cos (\xi)$ of degree at most 
\begin{equation}\label{q-rule:d}
D:= 2\bigl(m-\tilde{d}_{\tilde{\epsilon}}\bigr)-1.
\end{equation}
\end{subequations}
\end{theorem}
\begin{proof}
In the normalization from Eq. \eqref{bsp}, the orthogonality relations for the Bernstein-Szeg\"o polynomials read:
\begin{subequations}
\begin{equation}\label{co:a}
\int_0^\pi    \texttt{p}_l(\xi) \texttt{p}_k(\xi) w (\xi ) \text{d} x = \begin{cases}
\Delta_l^{-1} &\text{if}\ k=l,\\
0&\text{if}\ k\neq l,
\end{cases}
\end{equation}
for $l,k=0,1,2,\ldots $, where
\begin{equation}\label{co:b}
w(\xi)=\frac{1}{2\pi | c(\xi)|^2}=    \frac{2^{\epsilon_++\epsilon_-}(1+\epsilon_+\cos(\xi))   (1-\epsilon_-\cos(\xi)) }{2\pi \prod_{1\leq r\leq d} \bigl(1+2\alpha_r\cos (\xi )+\alpha_r^2\bigr)} .
\end{equation}
\end{subequations}
On the other hand, the discrete orthogonality in Eq. \eqref{ort:a} guarantees that
\begin{subequations}
\begin{equation}\label{do:a}
\sum_{0\leq \hat{l}\leq m}   \texttt{p}_l\bigl( \xi^{(m)}_{\hat{l}}\bigr) \texttt{p}_k\bigl( \xi^{(m)}_{\hat{l}} \bigr) \left| c \bigl(  \xi^{(m)}_{\hat{l}} \bigr) \right|^{-2}  
\hat{\Delta}^{(m)}_{\hat{l}} =
 \begin{cases}
\Delta_l^{-1} &\text{if}\ k=l,\\
0&\text{if}\ k\neq l,
\end{cases}
\end{equation}
for $0\leq l,k\leq m$ such that
\begin{equation}\label{do:b}
\begin{cases}
 0\leq l,k\leq  m-\lceil \tilde{d}_{\tilde{\epsilon}}\rceil &\text{if}\ \tilde{d}_{\tilde{\epsilon}}  \not\in\mathbb{Z} , \\
  0\leq l\leq  m-\lceil \tilde{d}_{\tilde{\epsilon}}\rceil 
  \ \text{and}\  0\leq k<  m-\lceil \tilde{d}_{\tilde{\epsilon}}\rceil 
  &\text{if}\   \tilde{d}_{\tilde{\epsilon}}  \in\mathbb{Z} .
  \end{cases}
  \end{equation}
  \end{subequations}
Upon comparing the orthogonality relations in Eqs. \eqref{co:a}, \eqref{co:b} and Eqs. \eqref{do:a}, \eqref{do:b},  it is clear that the exact
quadrature rule in Eqs. \eqref{q-rule:a}--\eqref{q-rule:c}
is valid if we pick $f\bigl(\cos (\xi)\bigr)=  \texttt{p}_l (\xi) \texttt{p}_k( \xi )$ with $0\leq l,k\leq m$ as in Eq. \eqref{do:b}. Since the products in question span the space of polynomials in $\cos (\xi)$ of degree at most $D$ \eqref{q-rule:d} (unless $\tilde{d}_{\tilde{\epsilon}}=-1$), the asserted quadrature formula follows in this situation by linearity. If $\tilde{d}_{\tilde{\epsilon}}=-1$, then the present arguments only recover the stated quadrature rule for polynomials $f\bigl( \cos(\xi)\bigr)$ up to degree $2m$. However, this case corresponds to that of the classical Gauss quadrature and is thus known to extend up to degree $D=2m+1$ (cf. the note below).
\end{proof}

\begin{note}
i). If $\tilde{d}=0$ and $\tilde{\epsilon}_\pm =1$ (so $\tilde{d}_{\epsilon}=-1$ and $D=2m+1$),  then Theorem \ref{quadrature:thm} amounts to the conventional Gaus quadrature formula
associated with the Bernstein-Szeg\"o polynomials supported on the roots of $\texttt{p}_{m+1}(\xi)$ \cite{dar-gon-jim:quadrature,bul-cru-dec-gon:rational,die-ems:exact}.
The case  $\tilde{d}=1$ and $\tilde{\epsilon}_\pm =1$ (so $\tilde{d}_{\epsilon}=-\frac{1}{2}$ and $D=2m$) provides, on the other hand, a corresponding
Gauss-type quadrature
supported on the roots of the quasi-orthogonal polynomial $\texttt{p}_{m+1}(\xi)+\tilde{\alpha}_1 \texttt{p}_m (\xi)$, cf. Refs.  \cite{mic-riv:numerical,gau:survey,ask:positive}.

ii). If $\epsilon_-=\tilde{\epsilon}_-=0$ and $\epsilon_+ + \tilde{\epsilon}_+ \neq 0$ (so $\tilde{d}_{\tilde{\epsilon}}=\frac{\tilde{d}-\tilde{\epsilon}_+}{2}$ and $D=2m+\tilde{\epsilon}_+ -\tilde{d} -1$)   or  if $\epsilon_- + \tilde{\epsilon}_-\neq 0$ and $\epsilon_+=\tilde{\epsilon}_+ = 0$ (so $\tilde{d}_{\tilde{\epsilon}}=\frac{\tilde{d}-\tilde{\epsilon}_-}{2}$ and $D=2m+\tilde{\epsilon}_- -\tilde{d} -1$),
then we arrive at Radau-type quadrature rules with a left-boundary node $x_0^{(m)}=0$ or a right-boundary node $x_m^{(m)}=\pi$, respectively, 
cf. Refs.  \cite{gau:survey,dav-rab:methods,gau:generalized,jou-bec:gautschi,pet:generalized}. If $\epsilon_\pm=\tilde{\epsilon}_\pm=0$ (so $\tilde{d}_{\tilde{\epsilon}}=\frac{\tilde{d}}{2}$ and $D=2m-\tilde{d}-1$), then both  boundary nodes $x_0^{(m)}=0$ and $x_m^{(m)}=\pi$ are present, whence our quadrature  is of Lobatto type in this situation, cf. Refs.   \cite{gau:survey,dav-rab:methods,gau:generalized,jou-bec:gautschi,pet:generalized}.
 It is important to emphasize at this point that the particular kinds of generalized Radau-- and Lobatto-type quadratures appearing here  do not involve function evaluations of the derivatives at the end-nodes (in contrast to those studied in  Refs. \cite{gau:generalized,jou-bec:gautschi,pet:generalized}).

 iii). Clearly the degree of exactness $D$ \eqref{q-rule:d} becomes unacceptably low if $\tilde{d}$ is too large. By requiring our quadrature to be interpolatory, i.e. $D\geq m$
(thus giving rise to the exact integration of interpolation functions $R(\xi)$ \eqref{q-rule:c} with $\deg (f) \leq m$ and (arbitrarily) prescribed values on the nodes),
one has to pick $2\tilde{d}_{\tilde{\epsilon}}\leq m -1$. In this situation,  the expansion of the characteristic polynomial $Q_{m+1}(\xi)$ \eqref{node-pol:a}, \eqref{node-pol:b} in the
orthogonal Bernstein-Szeg\"o basis is of the form  \cite{mic-riv:numerical,gau:survey,ask:positive,peh:linear,xu:characterization,ezz-gue:fast}
\begin{subequations}
\begin{equation}
Q_{m+1}(\xi)=\texttt{p}_{m+1}(\xi) + \sum_{1\leq k\leq 2(\tilde{d}_{\tilde{\epsilon}}+1)} c_{2m+1-k} \,\texttt{p}_{m+1-k}(\xi) ,
\end{equation}
for certain coefficients $c_{2m+1-k}\in\mathbb{R}$ ($k=1,2,\ldots ,2(\tilde{d}_{\tilde{\epsilon}}+1)$).  In other words, $Q_{m+1}(\xi)$ is orthogonal to $\texttt{p}_{m+1-k}(\xi)$
if $2(\tilde{d}_{\tilde{\epsilon}}+1)<k\leq m+1$, which is readily seen upon integrating the product
$f\bigl (\cos (\xi))=  Q_{m+1}(\xi)\texttt{p}_{m+1-k}(\xi)$ against the weight function $w(\xi)$ \eqref{bs-wf}, \eqref{c-f} by means of the quadrature rule of Theorem \ref{quadrature:thm}.
 If in addition  $2\tilde{d}_{\tilde{\epsilon}}\leq m -1 -d_{\epsilon}$  (so $m+1-k\geq d_{\epsilon}$ for $1\leq k\leq 2(\tilde{d}_{\tilde{\epsilon}}+1)$), then a comparison with Eq. \eqref{node-pol:a} reveals that in this case one has explicitly:
 \begin{equation}
 c_{2m+1-k}=\text{e}_k(\tilde{\alpha}; 1-\tilde{\epsilon}_+ ; 1-\tilde{\epsilon}_- ) .
   \end{equation}
   \end{subequations}
\end{note}


\bibliographystyle{amsplain}

\end{document}